\documentclass[12pt]{amsart}
\newtheorem{theorem}{Theorem}[section]
\newtheorem{lemma}[theorem]{Lemma}

\newcommand{\bC}{\mathbb{C}}

\newcommand{\bQ}{\mathbb{Q}}
\newcommand{\bR}{\mathbb{R}}
\newcommand{\bZ}{\mathbb{Z}}
\newcommand{\bF}{\mathbb{F}}

\newcommand{\bI}{\mathbb{I}}

\newcommand{\GL}{\mathrm{GL}}
\newcommand{\M}{\mathrm{M}}

\newcommand{\disp}{\displaystyle}

\def\disp{\displaystyle}
 \numberwithin{equation}{section}
\begin{document}
\title[title for the running head]{Title}
\baselineskip=17pt
\title[A note on non-Robba $p$-adic differential equations]
{A note on non-Robba $p$-adic differential equations}
\author{ Said Manjra}
{\curraddr{
Dept of  Math, Imam University.
P.Box 90950. Riyadh. 11623
Saudi Arabia}
\email{smanjra@uottawa.ca}
\thanks{}
\dedicatory{}
\maketitle
\begin{abstract}
Let $\mathcal{M}$ be a differential module, whose coefficients are analytic elements on an open  annulus $I$
($\subset \bR_{>0}$) in a valued field, complete and algebraically closed of inequal characteristic,
and let $R(\mathcal{M}, r)$ be the radius of convergence of its solutions in the
neighbourhood of the generic point $t_r$ of absolute value $r$, with $r\in I$. Assume that $R(\mathcal{M}, r)<r$ on $I$ and,
 in  the logarithmic coordinates, the function $r\longrightarrow R(\mathcal{M}, r)$ has only one slope on $I$.
 In this paper, we prove that for any
$r\in I$, all the solutions of $\mathcal{M}$ in the neighborhood of
 $t_r$  are analytic and
bounded in the disk $D(t_r,R(\mathcal{M},r)^-)$.
\end{abstract}
\section{Notations and Preliminaries}
Let $p$ be a prime number, $\bQ_p$ the completion of the field of
rational numbers for the $p$-adic absolute value $|.|$,  $\bC_p$ the
completion of the algebraic closure of $\bQ_p$, and $\Omega_p$ a
$p$-adic complete and algebraically closed field containing $\bC_p$
such that its value group is ${\bR}_{\ge 0}$ and the residue class
field is strictly transcendental over ${\bF}_{p^{\infty}}$.  For any
positive real $r$, $t_r$ will denote a generic point of
$\mathrm{\Omega}$ such that $|t_r|=r$. Let $I$ be a bounded interval
in ${\bR}_{>0}$. We denote by $\mathcal{A}(I)$ the ring of analytic
functions, on the annuli $\mathcal{C}(I):=\{a\in \Omega_p\;|\; |a|\in
I\}$, $\mathcal{A}(I)=\Big\{\disp\sum_{n\in{\bZ} }a_nx^n\in
\bC_p[[x,1/x]]\;\Big|\;\lim_{n\to\mp\infty}|a_n|r^n=0, \forall\; r\in
I\Big\},$ and by $\mathcal{H}(I)$ the completion of the ring of rational
fractions $f$ of $\bC_p(x)$  having no pole in $\mathcal{C}(I)$ with
respect to the norm $\|f\|_{I}:=\disp\sup_{r\in I}|f(t_{r})|$. It is well
known that $\mathcal{H}(I)\subseteq \mathcal{A}(I)$, with equality
if $I$ is closed. We define, for any $r \in I$, the absolute value $|.|_r$ over
$\mathcal{A}(I)$ by $\Big|\disp\sum_{n\in{\bZ}
}a_nx^n\Big|_r=\sup_{n\in{\bZ}}|a_n|r^n$.

Let $R(I)$ denotes $\mathcal{A}(I)$ or $\mathcal{H}(I)$. A free $R(I)$-module
$\mathcal{M}$ of finite rank $\mu$ is said to be
$R(I)$-differential module if it is equipped with a $R(I)$-linear map
$D:\mathcal{M}\to \mathcal{M}$ such that $D(am)=\partial(a)m+aD(m)$
for any $a\in R(I)$ and any $m\in \mathcal{M}$   where
$\partial=d/dx$.
 To each $R(I)$-basis $\{e_{i}\}_{1\le i\le \mu}$ of $\mathcal
   M$ over $R(I)$  corresponds a matrix $G=(G_{ij})\in {M}_\mu(R(I))$
  satisfying  $D(e_{i})=\disp\sum_{j=1}^{\mu}G_{ij}e_{j}$, called the matrix of $\partial$
   with respect to the $R(I)$-basis $\{e_{i}\}_{1\le i\le \mu}$ or simply an associated matrix to $\mathcal M$, together with a
  differential system $\partial X=GX$ where $X$ denotes a column
  vector $\mu\times 1$  or $\mu\times \mu$ matrix (see \cite{C}, \cite{CD}).
  If $G'\in {M}_\mu(R(I))$ is the matrix of $\partial$ with respect to another $R(I)$-basis $\{e'_{i}\}_{1\le i\le \mu}$  of $\mathcal M$
  and if $H=(H_{ij})\in {\GL}_\mu(R(I))$ is the change of basis matrix defined by
  $e'_i=\disp\sum_{i=1}^\mu H_{ij}e_i$ for all $1\le i\le \mu$, it is  known that:
  \begin{itemize}
     \item[-] the matrices  $G$ and $G'$ are related by the formula
  $G'=HGH^{-1}+\partial(H)H^{-1}.$
The matrix $HGH^{-1}+\partial(H)H^{-1}$ is denoted $H[G]$.
     \item[-]  if $Y$ is a solution matrix for the  system $\frac{d}{dx}X=GX$ with coefficients in a differential field extension of $R(I)$,  then the matrix $HY$ is a solution matrix for $\frac{d}{dx}X=H[G]X$.
   \end{itemize}
\subsection{Generic radius of convergence.}
Let ${\mathcal M}$ be an $R(I)$-differential module of rank $\mu$,
 $G=(G_{ij})\in {\M}_\mu(R(I))$  an associated matrix to
$\mathcal{M}$, $(G_{n})_n$ a sequence of matrices defined by
$$G_0=\bI_\mu\quad \text{and}\quad
G_{n+1}=\partial(G_{n})+G_{n}G,$$  and $||G||_r=\disp\max_{ij}|G_{ij}|_r$ be the norm of $G$ associated to the absolute value $|.|_r$.
For any $r\in I$, the quantity $R(\mathcal{M},r)=\min(r,\disp\liminf_{n\to\infty}||G_n||_r^{-1/n})$
 represents the radius of
convergence in the generic disc $D(t_r,r^-)$ of the solution matrix
$$\mathcal{U}_{G, t_r}(x)=\sum_{n\ge 0}\frac{G_n(t_r)}{n!}(x-t_r)^n$$ of
the system $\frac{d}{dx}X=GX$ with $X(t_r)=\bI_\mu$. We know
that the function $r\rightarrow R(\mathcal{M},r)$ is independent of the choice of basis and the ring $R(I)$
\cite[Proposition 1.3]{CD}, and the
graph of the  map $\rho\mapsto \log \circ R(\mathcal{M},\exp
(\rho))$, on any closed subinterval of $I$, is a  concave polygon
with rational slopes \cite[Theorem 2]{P}. This graph is called
the generic polygon of the convergence of $\mathcal{M}$. The system
 $\partial X=GX$ is said to have an analytic and bounded solution in
the disk $D(t_r,R(\mathcal{M},r)^-)$ if
$$\sup_{n\ge 0}\Big|\Big|\frac{G_n}{n!}\Big|\Big|_{r}R(\mathcal{M},r)^n<\infty.$$ The
$R(I)$-differential module $\mathcal{M}$ is said to be non-Robba if
$R(\mathcal{M},r)<r$ for all $r\in I$.
\subsection{Frobenius.} Let $\varphi:
\mathcal{C}(I)\to\mathcal{C}(I^p)$ be the Frobenius ramification
$x\mapsto x^p$, where $I^p$ is the image of $I$ by the map $x\mapsto
x^p$. A $R(I^p)$-differential module $\mathcal{N}$ is said to be a
Frobenius antecedent of an $R(I)$-differential module $\mathcal{M}$
if $\mathcal{M}$ is isomorphic to the inverse image
$\varphi^*\mathcal{N}$ of $\mathcal{N}$. In other words, if there
exists a matrix $F\in {\M}_\mu(R(I^p))$ of the derivation $d/dz$
(where $z=x^p$) in some $R(I^p)$-basis of $\mathcal{N}$ such that $p
x^{p-1}F(x^p)$ is a matrix of $d/dx$ in some $R(I)$-basis of
$\mathcal{M}$. The existence of such a Frobenius antecedent depends
of the values of the function $r\mapsto R(\mathcal{M},r)$. Recall
the Frobenius structure theorem of Christol-Mebkhout \cite[Theorem
4.1-4]{CM} where $\pi=p^{-1/p-1}$:
\begin{theorem}
\label{fro} Let $h$ be a positive integer and let $\mathcal{M}$ be
a $R(I)$-differential module such that $R(\mathcal{M},r)>r
\pi^{1/p^{h-1}}$ for all $r\in I$. Then,
 there exists an
$R(I^{p^h})$-differential module $\mathcal{N}_h$ such that
$(\varphi^h)^*{\mathcal{N}_h}\cong {\mathcal{M}}$ and
$R(\mathcal{M},r)^{p^h}=R(\mathcal{N}_h,r^{p^h})$ for any $r\in I$,
and  $\mathcal{N}_h$ is called a Frobenius  antecedent of order $h$ of $\mathcal{M}$.
\end{theorem}
In particular, if a $R(I)$-differential module $\mathcal{M}$ satisfies $R(\mathcal{M},r)>r
\pi$ for all $r\in I$, it has a Frobenius antecedent.
\section{Main theorem}
In this section, $I$ denotes an open interval in $\bR_{>0}$ and $\mathcal{M}$  a non-Robba
$\mathcal{A}(I)$-differential module associated to some matrix $G\in
\M_\mu(\mathcal{A}(I))$.
\begin{theorem}
Assume that  the generic polygon of convergence of $\mathcal{M}$ has only
one slope. Then $$\disp\sup_{n\ge
0}\Big\|\frac{G_n}{n!}\Big\|_{r}R(\mathcal{M},r)^n<\infty\quad
\text{for all}\quad r\in I.$$
\end{theorem}
The proof of this theorem   is based on the following lemmas:
\begin{lemma}\label{bound-fro}
Assume  $R(\mathcal{M},r)> \pi r$ for all $r\in I$ and let
$\mathcal{N}$ be a Frobenius  antecedent of $\mathcal{M}$. Let $F$
be an associated matrix to $\mathcal{N}$ and assume there exists a
real $r_0\in I$ such that $\disp\sup_{n\ge
0}||\frac{F_n}{n!}||_{r_0^p}R(\mathcal{N},r_0^p)^n<\infty$. Then
$\disp\sup_{n\ge 0}||\frac{G_n}{n!}||_{r_0}R(\mathcal{M},r_0)^n<\infty$.
\end{lemma}
\begin{proof}
The matrix
$\disp\mathcal{V}(z)=(\mathcal{V}_{ij}(z))_{ij}=\mathcal{V}_{F,t_{r_0}^p}(z)=\sum_{n\ge
0}\frac{F_n(t_{r_0}^p)}{n!}(z-t_{r_0}^p)^n$ is the solution matrix of
the differential system $\frac{d}{dz}V(z) =F(z)V(z)$ in the neighborhood of
$t_{r_0}^p$ with $z=x^p$ and $\mathcal{V}(t_{r_0^p})=\bI_\mu$. The change of variables leads to
$\frac{d}{dx}\mathcal{V}(x^p)=px^{p-1}F(x^p)\mathcal{V}(x^p)$. In
addition, since $R(\mathcal{M},r_0)> \pi r_0$, the map $x\mapsto
x^p$ sends the closed disk $D(t_{r_0},R(\mathcal{M},r_0))$ into
$D(t_{r_0}^p,R(\mathcal{M},r_0)^p)=D(t_{r_0}^p,R(\mathcal{N},r_0^p))$
\cite[Lemma 3.1]{BC}, and $$\sup_{n\ge
0}\Big\|\frac{F_n(t_{r_0}^p)}{n!}\Big\|.|x^p-t_{r_0}^p|^n=\sup_{n\ge
0}\Big\|\frac{F_n(t_{r_0}^p)}{n!}\Big\|.|x^{p-1}+x^{p-1}t_{r_0}+\ldots+t_{r_0}^{p-1}|^n.|x-t_{r_0}|^n<\infty$$
for all $x\in D(t_{r_0},R(\mathcal{M},r_0))$. In the neighborhood of
$t_{r_o}$, the matrix $\mathcal{V}_{F,t_{r_0}^p}(x^p)$ can be written
as $\mathcal{V}(x^p)=\disp\sum_{n\ge 0} B_n(x-t_{r_0})^n$ where
$B_n=(B_n(i,j))_{ij}$ are $\nu\times\nu$ matrices with entries un
$\Omega$. In that case,  we have $\disp\lim_{n\to
\infty}|B_n(i,j)|\rho^n=0$ for any $\rho <R(\mathcal{M},r_0)$, and therefore
\begin{equation}
\label{eqboun1} \sup_{n\ge 0}|B_n(i,j)|\rho^n=\sup_{x\in
D(t_{r_0},\rho)}|\mathcal{V}_{ij}(x^p)|\le \sup_{z\in
D(t_{r_0}^p,\rho^p)}|\mathcal{V}_{ij}(z)|=\sup_{n\ge
0}\Big\|\frac{F_n(t_{r_0}^p)}{n!}\Big\|\rho^{pn}.
\end{equation}
Since the matrix $\mathcal{V}(z)$ is  analytic and bounded in
$D(t_{r_0}^p,R(\mathcal{N},r_0^p)^-)$, there exists a positive
constant $C>0$, by  \cite[Proposition 2.3.3]{C}, such that
\begin{equation}
\label{eqboun2} \sup_{n\ge
0}\Big\|\frac{F_n(t_{r_0}^p)}{n!}\Big\|\rho^{pn}<C
\end{equation}
 for any
$\rho<R(\mathcal{M},r_0)$  and close to $R(\mathcal{M},r_0)$. Combining \eqref{eqboun1} and
\eqref{eqboun2}, and using again \cite[Proposition 2.3.3]{C}, we
find $\sup_{n\ge 0}|B_n(i,j)|R(\mathcal{M},r_0)^n<\infty$ for all $1\le i,j\le
\nu$, and therefore, the matrix $\mathcal{V}(x^p)$ is analytic and
bounded in the disk $D(t_{r_0},R(\mathcal{M},r_0)^-)$. In addition,
since the matrix $px^{p-1}F(x^p)$ is associated to $\mathcal{M}$,
then there exists an invertible matrix $H\in
\GL_\mu(\mathcal{A}(I))$ (hence $H$ is analytic and bounded in the
disk $D(t_{r_0},R(\mathcal{M},r_0)^-)$) such that
$G=H[px^{p-1}F(x^p)]$. Thus, by \cite[Proposition 2.3.2]{C}, the
matrix $H\mathcal{V}(x^p)$ is a solution to the system $\partial X=GX$
in the  neighborhood of $t_{r_0}$, and moreover it is analytic and
bounded in the disk $D(t_{r_0},R(\mathcal{M},r_0)^-)$. This means
that $\mathcal{U}_{G,t_{r_0}}(x)= H\mathcal{V}(x^p)H(t_{r_0})^{-1}$ is also analytic and bounded in the disk
$D(t_{r_0},R(\mathcal{M},r_0)^-)$.
\end{proof}
\begin{lemma}\label{density}
The set of reals $r$ in $I$ for which $\sup_{n\ge
0}\|\frac{G_n}{n!}\|_{r}R(\mathcal{M},r)^n<\infty$ is dense in $I$.
\end{lemma}
\begin{proof}
Let $J$ be a closed subinterval of $I$ not reduced to a point and let $\rho$ be a real number in the interior of $J$.
 Then, by hypothesis, $R(\mathcal{M},\rho)/\rho <1$ and therefore there exists an  integer $h$ such that $\pi^{1/p^{h-1}}<R(\mathcal{M},\rho)/\rho<\pi^{1/p^{h}}$.
Since the function $r\mapsto R(\mathcal{M},r)$ is continuous on $J$, there exists
 an open subinterval $J'\subset J$ containing $\rho$
such that $\pi^{1/p^{h-1}}r<R(\mathcal{M},r)<\pi^{1/p^{h}}r$ for all $r \in J'$. \\
There are two cases to consider:\\
\textbf{Case 1: $h\le 0$.}\\ Let $\dot{\mathcal{H}}(J')$ be the
quotient field of $\mathcal{H}(J')$. By cyclic vector lemma, we can
associate $\dot{\mathcal{H}}(J')\otimes\mathcal{M}$  to a
differential equation
$\Delta(\dot{\mathcal{H}}(J)\otimes\mathcal{M})=
\partial^\mu+q_1(x)\partial^{\mu-1}+\ldots+q_\mu(x)$, where
$q_i\in \dot{\mathcal{H}}(J')$ for $i=1,\ldots,\mu$.  Now pick a
nonempty subinterval $J''$ of $J'$ such that $q_i\in
{\mathcal{H}}(J'')$ for $i=1,\ldots,\mu$, and let  $r_0$ be a real number
in the interval $J''$ and $\lambda(r_0)$ be the maximum of the $p$-adic absolute
values of the roots of the polynomial
$\Delta(\dot{\mathcal{H}}(J)\otimes\mathcal{M})=
\lambda^\mu+q_1(t_{r_0})\lambda^{\mu-1}+\ldots+q_\mu(t_{r_0})$.
Since
$R(\mathcal{M},r_0)=R(\dot{\mathcal{H}}(J)\otimes\mathcal{M},r_0)<\pi^{1/p^{h}}r_0<\pi
r_0$, by virtue of  \cite[Theorem 3.1]{Y}, we have $\log
(R(\mathcal{M},r_0))=\frac{1}{p-1}+\log(\lambda(r_0))$ and all the
solutions $u_1,\ldots,u_\mu$ of
$\Delta(\dot{\mathcal{H}}(J)\otimes\mathcal{M})$ in the neighborhood
of $t_{r_0}$ are analytic  and bounded in the disk
$D(t_{r_0},R(\mathcal{M},r_0)^-)$. Now let $W$ be the wronskian matrix
of $(u_1,\ldots,u_\mu)$. Then, $W$ is a solution of the system
$$\partial
X=A_{\Delta(\dot{\mathcal{H}}(J)\otimes\mathcal{M})}X\;\;\mbox{where}\;\;
A_{\Delta(\dot{\mathcal{H}}(J)\otimes\mathcal{M})}:= \left (
\begin{array}{c}
0\\0\\\vdots\\0\\-q_{\mu}
\end{array}
\begin{array}{c}
1\\0\\\vdots\\0\\-q_{\mu-1}
\end{array}
\begin{array}{c}
0\\1\\\vdots\\0\\-q_{\mu-2}
\end{array}
\begin{array}{c}
\ldots\\\ldots\\\ddots\\ \\\ldots
\end{array}
\begin{array}{c}
0\\0\\\vdots\\1\\-q_{1}
\end{array}
\right).$$ Moreover, by  \cite[Proposition 2.3.2]{C}, the matrix
$W$ is analytic and bounded in the disk
$D(t_{r_0},R(\mathcal{M},r_0)^-)$. Since $G$ and
$A_{\Delta(\dot{\mathcal{H}}(J)\otimes\mathcal{M})}$ are associated
to ${\mathcal{H}}(J'')\otimes\mathcal{M}$, there exists a
matrix $H\in \GL_\mu({\mathcal{H}}(J''))$ such that
$G=H[A_{\Delta(\dot{\mathcal{H}}(J)\otimes\mathcal{M})}]$. Since
$R(\mathcal{M},r_0)<r_0$, the matrix $H$  is analytic and bounded in
the disk $D(t_{r_0},R(\mathcal{M},r_0)^-)$. Hence, by \cite[Proposition 2.3.2]{C}, the matrix
$\mathcal{U}_{G,t_{r_0}}(x)=HWH(t_{r_0})^{-1}$ is
also analytic and bounded in the disk
$D(t_{r_0},R(\mathcal{M},r_0))$. This ends the proof of the lemma in
this case.\\
\textbf{Case 2: $h>0$.}\\
Applying Theorem \ref{fro} to ${\mathcal{H}}(J')\otimes\mathcal{M}$,
there exists a $\mathcal{H}(J'^{p^h})$-differential module
$\mathcal{N}_h$ which is  a Frobenius antecedent of order $h$ of
${\mathcal{H}}(J')\otimes\mathcal{M}$. Moreover,
$R(\mathcal{N}_h,\rho)<\pi \rho$ for all $\rho\in J'^{p^h}$. Let
$^hF$ be an associated matrix of $\mathcal{N}_h$. Then, by case 1,
there exists $r_0\in J'$ such that $^hF$ is  analytic and bounded in
the disk $D(t_{r_0}^{p^h},R(\mathcal{N}_h,r_0^{p^h}))$. The proof of
the lemma in this case can be concluded by iteration of Lemma
\ref{bound-fro}.
\end{proof}
\begin{proof}[Proof of Theorem 2.1] By hypothesis, the generic polygon of convergence of
 $\mathcal{M}$ has only one slope. This slope is a rational number by \cite[Theorem 2]{P}. Thus, we may assume there exist $\alpha\in \bC_p$ and
$\beta\in \bQ$ such that $R(\mathcal{M},r)=|\alpha|r^\beta$ for all
$r\in I$. \\
Let now $r$ be a real in the interior of $I$. Then,  by Lemma
\ref{density}, there exist two reals $r_1, r_2\in I$ such that
$r_1<r<r_2$ and
$$\sup_{n\ge
0}\|\frac{G_n}{n!}\|_{r_1}R(\mathcal{M},r_1)^n<\infty\quad
\text{and}\quad\sup_{n\ge
0}\|\frac{G_n}{n!}\|_{r_2}R(\mathcal{M},r_2)^n<\infty,$$ which are
equivalent to
$$\sup_{n\ge
0}\|\frac{G_n}{n!}\alpha^n x^{n\beta}\|_{r_1}<\infty\quad
\text{and}\quad\sup_{n\ge 0}\|\frac{G_n}{n!}\alpha^n
x^{n\beta}\|_{r_2}<\infty.$$ Since all  the matrices $\alpha^n
x^{n\beta}G_n$ have all their entries in $\mathcal{H}[r_1,r_2]$, and
for any  element $f\in \mathcal{H}([r_1,r_2])$, we have
$|f|_r\le\max(|f|_{r_1},|f|_{r_2})$,  then  for any integer $n\ge
0$, we have
\begin{multline*}
\|\frac{G_n}{n!}\|_{r}R(\mathcal{M},r)^n 
\le\|\frac{G_n}{n!}\alpha^n x^{n\beta}\|_{r}\\
\quad\le
\max(\|\frac{G_n}{n!}\alpha^n
x^{n\beta}\|_{r_1},\|\frac{G_n}{n!}\alpha^n x^{n\beta}\|_{r_2})\\
\le \max(\sup_{n\ge 0}\|\frac{G_n}{n!}\alpha^n
x^{n\beta}\|_{r_1},\sup_{n\ge 0}\|\frac{G_n}{n!}\alpha^n
x^{n\beta}\|_{r_2}).
\end{multline*}
Hence,   $$\sup_{n\ge 0}\|\frac{G_n}{n!}\|_{r}R(\mathcal{M},r)^n\le
\max(\sup_{n\ge 0}\|\frac{G_n}{n!}\alpha^n
x^{n\beta}\|_{r_1},\sup_{n\ge 0}\|\frac{G_n}{n!}\alpha^n
x^{n\beta}\|_{r_2})<\infty.$$
\end{proof}

\end{document}